\documentclass{amsart}
\usepackage{amsfonts,amssymb,amsmath,amsthm}
\usepackage{url}
\usepackage{enumerate}

\newtheorem{thrm}{Theorem}[section]
\newtheorem{lem}[thrm]{Lemma}
\newtheorem{prop}[thrm]{Proposition}
\newtheorem{cor}[thrm]{Corollary}
\theoremstyle{definition}

\numberwithin{equation}{section}

\def \bC {\mathbb C}

\def \bF {\mathbb F}

\def \bH {\mathbb H}

\def \bR {\mathbb R}

\def \bZ {\mathbb Z}

\def \cC {\mathcal C}

\def \cF {\mathcal F}

\def \cO {\mathcal O}

\def \cS {\mathcal S}

\def \cV {\mathcal V}
\def \cW {\mathcal W}
\def \cX {\mathcal X}

\def \fa {\mathfrak a}

\def \fg {\mathfrak g}
\def \fh {\mathfrak h}

\def \fk {\mathfrak k}

\def \fm {\mathfrak m}
\def \fn {\mathfrak n}

\def \fp {\mathfrak p}

\def \ft {\mathfrak t}

\def \bfZ {\mathbf Z}

\def \Mp {\text{\rm Mp}}
\def \Sp {\text{\rm Sp}}
\def \SO {\text{\rm SO}}
\def \so {\mathfrak {so}}

\def \SU {\text{\rm SU}}

\def \SL {\text{\rm SL}}
\def \GL {\text{\rm GL}}
\def \gl {\mathfrak {gl}}
\def \Mat {\text{\rm M}}

\def \ad {\text{\rm ad}}

\def \Pf {\text{\rm Pf}}

\def \Tr {\text{\rm Tr}}
\def \diag {\text{\rm diag}}
\def \sgn {\text{\rm sgn}}

\def\bm{\underline{\bold m}}
\def\bn{\underline{\bold n}}

\begin{document}

\author{V. Fischer and G. Zhang}

\address{Mathematical Sciences, Chalmers University of Technology and
Mathematical Sciences, G\"oteborg University, SE-412 96 G\"oteborg, Sweden}

\email{genkai@chalmers.se}

\address{Università degli studi di Padova, DMMMSA, Via Trieste 63,
           35121 Padova, Italy}

\email{fischer@dmsa.unipd.it}



\thanks{G. Zhang acknowledges the  partially support by Swedish Research Council (VR). Veronique Fischer acknowledges the support of the London Mathematical Society (LMS)}

\title[Degenerate principal series representations of  $\SO(p+1, p)$]
{Degenerate Principal Series Representations of $\SO(p+1,p)$}

\begin{abstract}
For  $p$ odd,
the Lie
group $G^\sharp=\SO_0(p+1,p+1)$
 has a family of
unitary degenerate principal series representations
realized on the space of real $(p+1)\times (p+1)$
skew symmetric matrices, similar
to the Stein's complementary series
for $\SL(2n, \bC)$ or Speh's
representation for $\SL(2n, \bR)$. 
We consider their restriction on the subgroup
 $G_0=\SO_0(p+1,p)$ and prove
that they are still irreducible and is equivalent
to (a unitarization of)  the principal
series representation of $G$, and also irreducible
under a maximal parabolic subgroup of $G$.
\end{abstract}

\maketitle

\section{Introduction}

In the present paper we shall study
the unitarity of degenerate principal series
representations of the group $G=\SO(p+1, p)$
induced from certain maximal parabolic subgroup 
for odd $p=2q-1$.

In the case of the group $\SO_0(n, n)$, or more generally
$\SU(n, n; \bF)$, for $\bF
=\bR, \bC,\bH$, 
 Johnson 
\cite{Johnson-mathann} has determined 
the range of unitarity of the
representations;
 for $\SU(n,n,\bF)$, $n$ even, 
 he found certain complementary
series. 
Some generalizations
of these results were obtained in \cite{Sahi-crelle, oz-duke, gkz-ma,
Howe-Lee-jfa} for a larger class of groups
by using  computations in the compact picture. 
The analysis of  these
representations in the non-compact picture
has been done in \cite{BSZ}.

 We shall prove that the 
  restriction of the complementary series of
$G^\sharp=\SO_0(p+1, p+1)$ 
to the opposite maximal parabolic subgroup
of $G$  is irreducible,
and in particular the restriction
to the identity component $G_0$
of $G$ is irreducible.
 We shall use mostly
 the non-compact realization
of the principal series. 
The proof relies on both
Euclidean  and nilpotent Fourier
transform.

The restriction
of the degenerate principal series
representations of $\SO(n, n)$ 
to the subgroups $\SO(n, m)\times \SO(n-m)$ 
for $m <n$ has
been studied earlier by Lee-Loke \cite{Lee-Loke}
in the compact picture,
 the representations 
of $\SO(n, m)
\times \SO(n-m)$ appearing
are of the form $\tau \times \tau'$,
and the representations
$\tau$ are degenerate principal
series. It
might be true that the representations
$\tau$ of $\SO(n, m)$ are also irreducible
under corresponding  the maximal parabolic subgroup, as we show
here for $m=n-1$.
We mention also that there has been quite some study of
 complementary series
representations for semisimple Lie 
groups.  In \cite{BK}
 a large class
of complementary series representations is
 constructed 
 with parabolic subgroups being cuspidal and maximal;
our case of 
$\SO(p+1, p)$  here is however not cuspidal.
For the groups $\SO(p, q)$
some complementary series similar
to ours can be constructed
by using the branching of holomorphic
representations of $\SU(p,q)$ to
$\SO(p, q)$. However those constructions
works only for $q-p>2$
(see \cite[Therorem 3.1]{Neretin-Olshanski},
\cite[Therorem 5.2]{Z-aim}),
 and thus they do not cover
our present case of $\SO(p+1, p)$.
Our result about the irreducibility
of the restriction to $\SO(p+1, p)$
of representations of $\SO_0(p+1, p+1)$
is in a way similar to the Kirillov
conjecture, now a theorem \cite{Baruch-annmath, Sahi-Ste} for 
$\GL_{p+1}(\bR)$
and $\GL_p(\bR)\times \bR^p$. See also
\cite{Speh-V} on the study of the restriction
of complementary series of $\SO(n, 1)$
to $\SO(n-1, 1)$,   \cite{Kobayashi-korea}
on  branching of highest weight
representations, and 
\cite{Kobayashi-contmath}
 the classification of finitely decomposable
representations of $G^\sharp$ under $G$.
We note also that the irreducibility
result under the the parabolic group $P$
can possibly be also proved abstractly
by using the Mackey theory on induced
representations. However we present
a relatively  elementary and direct proof, in particular
it also yields a decomposition of
the representation under the subgroup $\bar N \Sp(q-1, \mathbb R)$.

\section{Preliminaries}
\label{sec_G_P}

\medskip
\subsection{The group 
 $G^\sharp=\SO_0(p+1,p+1)$ and 
$G=\SO(p+1,p)$}\hfill\\ \vspace{-0.5em}

Let  $\Mat_{p,q}$
be  the space of real $p\times q$-matrices,
and denote $\Mat_p= \Mat_{p, p}$.
Denote  $\cX=\cX_p=\{X\in \Mat_p; X=-X^t\}$ 
the subspace of  skew-symmetric
real matrices. 
We will also use
 the short-hand notation $X^{-t}={(X^t)}^{-1}$,
for an invertible  $X$.
Denote $I_n$ the identity matrix in $\Mat_n$ 
and $I_{p,q}=\diag(-I_p,I_q)$.

Let $p>1$ and $G^\sharp=
\SO_0(p+1,p+1)$
be the  identity component of $\SO(p+1, p+1)
=\{g\in \Mat_{2p+2}; \det g=1,\
 g^t I_{p+1,p+1} g=I_{p+1,p+1}\}
$,  
and $G=\SO(p+1,p)$
realized as the subgroup of $G^\sharp$ via:
$$
G=\{\diag(g,  1)\in G^\sharp\} \subset G^\sharp.
$$
The group $G$ has two connected components
and we denote  its identity component by $G_0$. 

Elements $g$ of $G^\sharp$ 
will be written as $2\times 2$ block matrices
$$
g=
\begin{pmatrix}
g_{1,1}&g_{1,2}
\\
g_{2,1}&g_{2,2}
\end{pmatrix},
$$ 
with each entry
being in $\Mat_{p+1}$.
The Lie algebra $\fg^\sharp$ of $G^\sharp$
has the   decomposition $\fg^\sharp=\fk^\sharp \oplus\fp^\sharp$
with $\fk^\sharp=\so(p+1)\oplus
\so(p+1)$
with respect to the
Cartan involution
 $g\to g^{-t}$.
The group 
$$
K^\sharp=\{\diag (k_1,k_2); \ k_1, k_2\in\SO(p+1)\}
=\SO(p+1)\times
\SO(p+1)\ ,
$$
is a maximal compact subgroup of $G^\sharp$
 with Lie algebra $\fk^\sharp$.
Correspondingly
$\fg=\fk \oplus\fp$,
$\fk= \so(p+1)\oplus
\so(p)$ and 
$$
K_0=\{\diag (k_1,k_2, 1); \ k_1\in\SO(p+1)  
\, , \ k_2\in {\rm SO}(p) \}
\sim\SO(p+1)  \times \SO(p)
 \,
 $$
is a maximal compact subgroup of $G_0$,
while
$$
K=\{\diag (k_1,k_2, \det k_2); \ k_1\in\SO(p+1)  
\, , \ k_2\in {\rm O}(p) \}
\sim\SO(p+1)  \times {\rm O}(p)
 \,
 $$
is a maximal compact subgroup of $G$.
Note that 
$K=\{I_{p+1,p+1},I_{2p+2}\} \times K_0\sim 
\bZ_2\times P$.

For $j=1,\cdots, p+1$,
let 
$$
H_j=\begin{pmatrix}
0&  X\\
X^t & 0
\end{pmatrix}\in \fp^\sharp 
\quad\mbox{where}\quad
 X=\diag(0, \cdots, 0, 1, 0, \cdots, 0)
\ ,
$$
with $1$ on the $j$th position.
Then $\ft^\sharp:=\bR H_1 +\cdots + \bR H_{p+1}$ 
and $\ft:=\bR H_1 +\cdots + \bR H_{p}$ 
are maximal abelian subspaces of $\fp^\sharp$ and $\fp$.
Let $\{\epsilon_j\}$ be the dual basis of 
$\{H_j\}$. The positive root systems of
$(\fg^\sharp, \ft^\sharp)$  and
$(\fg, \ft)
$  are $\{\epsilon_j  \pm \epsilon_k, \ 1\le j< k\le p+1\}$
and $\{\epsilon_j  \pm \epsilon_k, \ 1\le j< k\le p\}$.

\medskip

\subsection{The maximal parabolic subgroups $P^\sharp$
and $P$.}\hfill\\ \vspace{-0.5em}

We fix the elements:
$$
\xi^\sharp
=H_1 +\cdots +H_{p+1}\in \fp^\sharp
\quad \mbox{and}\quad
\xi
=H_1 +\cdots +H_{p}\in \fp
\ .
$$
Let $\fa^\sharp=\bR\xi^\sharp$, $\fa=\bR\xi$.
The root space decomposition
of $
\fg^\sharp
$ 
 under $\xi^\sharp$ and
$\fg$  under $\xi$ is then
$$
\fg^\sharp=\fn^\sharp_{-2}
 + \fm^\sharp +\fa^\sharp
+\fn_2^\sharp, 
\quad
\fg=\fn_{-2} + \fn_{-1}
+ \fm +\fa 
+\fn_{1}+\fn_{2} 
\ ,
$$
with roots $\pm{2}, 0$,
and 
 $\pm 2,\pm 1, 0$ respectively. 
We set
 for the corresponding positive root subspaces:
 $$
\fn^\sharp=\fn_2^\sharp
\quad\mbox{and}\quad
\fn=\fn_{1}+\fn_{2}
\ ,
 $$
and for the negative root spaces:
$$
\bar\fn^\sharp:=\fn_{-2}^\sharp=
(\fn_2^\sharp)^t
\quad\mbox{and}\quad
\bar \fn:= \fn_{-1}\oplus \fn_{-2} =\fn^t
\ .
$$

We shall use explicit forms for the root spaces:
$$
\fn^\sharp
=\{n^\sharp_{Z}, \, Z\in \cX_{p+1}\}
\sim\cX_{p+1}
\quad\mbox{where}\quad
n_{Z}:=
 \begin{pmatrix}
Z&-Z\\
Z&-Z\\
\end{pmatrix}
\ ,
$$
and
$$
\fn=\{
\diag(n_{(z,v)},0);  \,
z\in \mathcal X_{p}, \, v\in \bR^p\}
\quad\mbox{where}\quad
n_{(z,v)}:=
\begin{pmatrix}
z&v&-z\\
-v^t&0&v^t\\
z&v&-z
\end{pmatrix}
\ .
$$
We have:
$$
\fn=\fn_1 \oplus \fn_2
\quad\mbox{with}\quad
\fn_1=n_{0,\bR^p}\sim \bR^p
\ , \  \fn_2=n_{\cX_p,0}\sim \cX_p \ .
$$

The Lie algebra $\fn^\sharp$ is abelian and 
$\fn$ is a 2-step nilpotent Lie sub-algebra of $\fg$.
Elements
$n_{(z,v)}$ will simply be written as $(z,v)$. 
The Lie
bracket in $\fn$ is given, via the above
identification $\fn=\cX_p\oplus \bR^p$, is
$$
[ z_1+v_1, z_2+v_2]=v_1v_2^t -v_2v_1^t .
$$
Thus $\fn$ is the free nilpotent Lie algebra with $p$ generators (over $\bR$).

The centralizer of $\fa^\sharp$ in $\fg^\sharp$
is 
$$
\fm^\sharp\oplus\fa^\sharp=\{l^\sharp_{(X, Y)},
\, X=X^t,  Y^t=-Y\in  \Mat_{p+1}\},
\quad
l^\sharp_{(X, Y)}:=
 \begin{pmatrix}
Y&X\\
X&Y\\
\end{pmatrix}, 
$$
identified with $\gl(p+1)$ via 
$
l^\sharp_{(X, Y)}\mapsto X+Y\in \gl(p+1)$, 
whereas 
the centralizer of $\fa$ in $\fg$ is
$$
\fm\oplus\fa=\{\diag(l_{(x,y)},0);
\, x=x^t,  y^t=-y\in  \Mat_{p}
\}, \quad
l_{(x,y)}:=
 \begin{pmatrix}
y&0& x\\
0& 0 & 0\\
x&0& y\\
\end{pmatrix},
$$
identified with $\gl(p)$.
Note that 
\begin{equation}
  \label{ma-ma}
\gl(p)\sim 
\fm\oplus \fa 
\ \subset \ \fm^\sharp \oplus \fa^\sharp
\sim \gl(p+1) \ .
\end{equation}

Let $M^\sharp$, $A^\sharp$, $N^\sharp$ and $\bar N^\sharp$ 
be the  simply connected subgroup of $G^\sharp$ with 
Lie algebras
$\fm^\sharp$, $\fa^\sharp$, $\fn^\sharp$ and 
$\bar\fn^\sharp$ respectively. 
Let $P^\sharp=M^\sharp A^\sharp N^\sharp$ 
and $\bar P^\sharp=M^\sharp A^\sharp \bar N^\sharp$
be the  corresponding maximal parabolic subgroups
of $G^\sharp$. Similarly we define 
the connected subgroups 
$P_0=M_0AN$
and $\bar P_0=M_0A\bar N$
of $G_0$
with Lie algebra $\fm +\fa  +\fn$
and $\fm +\fa  +\bar \fn$.
Note that the centralizer
of $\fa^\sharp$ in $K^\sharp$ is 
$$
Z_{K^\sharp}(\fa^\sharp)
=K^\sharp\cap M^\sharp
=\{\diag (k_1, k_1); \ k_1\in \SO(p+1)\}\sim\SO(p+1)
\ ,
$$
and the centralizer of $\fa$ in $K$ is
$$
Z_K(\fa)
=
K\cap M =\{\diag (k_2, \det k_2,k_2,\det k_2);
 \ k_2\in {\rm O}(p)\}\sim  {\rm O}(p)
 \ ,
$$
and the centralizer of $\fa$ in $K_0$ is the connected component of the identity of $Z_K(\fa)$.
We set:
$$
M=Z_K(\fa) M_0
\quad,\quad
P= Z_K(\fa)P_0=MAN
\quad\mbox{and}\quad
\bar P=Z_K(\fa) \bar P_0=MA\bar N
\ .
$$

The group $M^\sharp A^\sharp$ 
is isomorphic to the matrix group 
$\GL_{p+1}^+=\{h\in \Mat_{p+1},\ \det h>0\}$
via:
\begin{equation}
\label{isomorphism_GL_MA}
h \in \GL_{p+1}^+  \longmapsto 
k_o^t
\begin{pmatrix}
h&0\\0&h^{-t}
\end{pmatrix}
k_o
\in M^\sharp A^\sharp
\quad,\quad
k_o=
2^{-\frac12}
\begin{pmatrix}
1&1\\-1&1
\end{pmatrix}
\ .
\end{equation}
Restricting this isomorphism to $\{\diag (h,\det h), \ h\in \GL_p\}$
and to $\{\diag (h,\det h), \ h\in \GL_p^+\}$,
we obtain an isomorphism between $MA$ and $\GL_p$
and between $M_0A$ and $\GL_p^+$.
Thus, using the isomorphisms just above as identifications,
$P^\sharp$, $P$ and $P_0$ can be described as the semi-direct product:
$$
P^\sharp= \GL^+_{p+1} N^\sharp
\quad,\quad
P=\GL_{p}N
\quad\mbox{and}\quad
P_0=\GL^+_{p}N
 \ .
 $$
 
\begin{lem}
\label{lem_inclusion}
The following inclusions of Lie algebras hold:
$$
\fm\oplus \fa \subset \fm^\sharp \oplus \fa^\sharp
\quad\mbox{and}\quad
\fm\oplus \fa \oplus \fn \subset \fm^\sharp \oplus \fa^\sharp \oplus \fn^\sharp
\ .
$$
We have:
$$
M_0A\subset M A\subset  M^\sharp A^\sharp
\quad\mbox{and}\quad
P_0\subset P\subset P^\sharp
\ .
$$
Moreover we have a
factorization of  $\exp(n_{(z,v)})\in N$
in $\GL_{p+1}^+ =M^\sharp A^\sharp$,
\begin{equation}
\label{relation_n_nsharp}
\exp n_{(z,v)}= m \exp n^\sharp_{M(z,v)}
\ ,
\end{equation}
where $m \in M^\sharp$ corresponds to $\begin{pmatrix} I_p& v \\ 0&1\end{pmatrix} \in \SL_{p+1}$ and 
\begin{equation}
\label{M(z,v)}
M(z,v):=\left[
\begin{array}{cc}
z&\frac 12 v\\-\frac 12 v^t&0
\end{array}\right]
\in \cX_{2q}
\ ,
\end{equation}
is viewed as an element of $\fn^\sharp$.
\end{lem}

\begin{proof} The first relation is in (\ref{ma-ma}). 
We can write
$n_{(z,v)}=l^\sharp_{(X,Y)}+n^\sharp_W$ with:
$$
X=
\begin{pmatrix}
0&\frac 12 v\\ \frac 12 v^t&0
\end{pmatrix},
\quad  \quad
Y=
\begin{pmatrix}
0&\frac 12 v\\-\frac 12 v^t&0
\end{pmatrix},
\quad
W=
\begin{pmatrix}
z&\frac 12 v\\-\frac 12 v^t&0
\end{pmatrix}
\ .
$$
This shows $\fn \subset \fm^\sharp  \oplus \fn^\sharp$
and  implies $\fm\oplus \fa \oplus \fn \subset \fm^\sharp \oplus \fa^\sharp \oplus \fn^\sharp$.
The group inclusions follow immediately. 
Easy matrix computations give the equality \eqref{relation_n_nsharp}.
\end{proof}

Note that the adjoint
action of $MA=\GL_p^+$ and $\GL_p$ 
on $N$ is
\begin{equation}
\label{ma-on-n}
g\cdot (z, v)=(gzg^t, gv) \ .
\end{equation}

\section{Degenerate principal series representations}
\label{sec_deg_princ_series}

Throughout the paper we assume $p$ is odd and we write $2q=p+1$.

\medskip

\subsection{Principal series 
of $G^\sharp=\SO_0(p+1, p+1)$}\hfill\\ \vspace{-0.5em}

Let $\mu\in \bC$
and
consider the induced representation $I^\sharp(\mu)$ of $G^\sharp$
from the following character on $P^\sharp$
$$
\chi_\mu^\sharp:
me^{t\xi^\sharp}n \longmapsto e^{(\mu+\rho^\sharp) t}
\ ,
$$
where $\rho^\sharp=q(2q-1)$
is the half sum of the positive root
of $\ad \xi^\sharp$. In terms of $P^\sharp=\GL^+_{2q} 
N^\sharp$
this is 
\begin{equation}
\label{chi-mu}
\chi_\mu^\sharp:
l\, n \longmapsto \det(l)^{\frac{\mu+\rho^\sharp}{2q}},
\quad l\in \GL^+_{2q}.
\end{equation}

This representation  can be realized 
on the space  of Haar measurable functions $f(g)$
on $G^\sharp$ such that
\begin{equation}
\label{Isharp-mu}
f(gl
n)=\det(l)^{-\frac{\mu+\rho^\sharp}{2q}}
f(g), 
\,\, l\,n\in  P^\sharp,
\end{equation}
and
\begin{equation}
\label{Isharp-mu-}
f{\big |}_{K^\sharp}\in L^2(K^\sharp)
\ .
\end{equation}
See \cite{Kn-book}.
The group  $G^\sharp$ acts 
on $I^\sharp(\mu)$
  by the  left regular action and we denote
the representation by $(I^\sharp(\mu),  \pi^\sharp_\mu)$.
The condition \eqref{Isharp-mu} implies that $f \in I^\sharp(\mu)$
is invariant under the right action
of $K^\sharp\cap M^\sharp$ and 
can therefore be  identified as  functions on 
$K^\sharp/K^\sharp\cap M^\sharp$.
However
 $K^\sharp/K^\sharp\cap M^\sharp$ can be realized as $\SO(2q)$
since  the group $K^\sharp$ acts on $\SO(2q)$ by 
$$
K^\sharp \ni\diag(k_1,k_2): 
\left\{
\begin{array}{ccl}
\SO(2q)&\longrightarrow& \SO(2q)\\
 a &\longmapsto& k_1 a k_2^{-1}
 \end{array}
\right.
 \ ,
$$
and the isotropy group of the identity matrix $I_{2q}\in \SO(2q)$ is $K^\sharp\cap M^\sharp$.
Thus condition \eqref{Isharp-mu-} can be equivalently replaced by
$$
f\big{|}_{K^\sharp/K^\sharp\cap M^\sharp}\in 
L^2(K^\sharp/K^\sharp \cap M^\sharp)
=L^2(\SO(2q))
 \ .
$$

We denote
by $(I^\sharp_{K^\sharp}(\mu),  \pi^\sharp_\mu )$ 
the space of $K^\sharp$-finite
elements.  We will need its decomposition under $K^\sharp$.
Recall, see e.g. \cite{Johnson-mathann}, that 
each irreducible representation  of $\SO(2q)$ is
determined by a $q$-tuple of integers:
$$
\bm=(m_1, \cdots, m_q), \quad m_1\ge \cdots \ge m_{q-1}\ge |m_{q}| \ .
$$
We write $\cV_{\bm}$
for the representation space of $\bm$. 
Thus $I^\sharp_{K^\sharp}(\mu)
$ is the same as
the space $L^2(\SO(2q))_{K^\sharp}$
of $K^\sharp$-finite elements in $L^2(\SO(2q))$
and  we have \cite{Johnson-mathann}:
\begin{equation}
  \label{eq:L-2-K-1}
 I^\sharp_{K^\sharp}(\mu)\sim
L^2(\SO(2q))_{K^\sharp}=
\sum_{\bm}\cV_{\bm}\otimes \cV_{\bm}^\ast,
\end{equation}
where $\cV_{\bm}^\ast$ stands for the dual representation
of $\cV_{\bm}$.

Each function in the space $I^\sharp(\mu)$ is also uniquely
determined by its restriction to $\bar N^\sharp$. The $G^\sharp$-action
in this realization is referred as $\bar N^\sharp$-realization.
It is \cite[p.~169]{Kn-book} the space
$L^2(\bar N^\sharp, e^{2\Re (\mu +\rho^\sharp)H^\sharp}  )$
where the function $H^\sharp$ is defined  by $t=H^\sharp(\bar n)$ 
using the Iwasawa decomposition of
$\bar n= kme^{t\xi ^\sharp}n_+ \in K^\sharp P^\sharp $.
Note however that
the $L^2$-norm in $L^2(K)$ or $L^2(\bar N^\sharp,e^{2\Re (\mu +\rho^\sharp)H^\sharp} )$ is  
$G^\sharp$-invariant only
for purely imaginary $\mu$, $\mu\in i\bR$.

\medskip

\subsection{Zeta distribution and complementary series 
$\mathcal C^\sharp(\nu)$ of $G^\sharp$
}\label{subsec_zeta}\hfill\\ \vspace{-0.5em}

The unitarity of  $(I^\sharp(\mu),  \pi_\mu )$  for $\mu$ outside the
standard unitary range $\mu\in i\bR$ has
been completely determined in \cite{Johnson-mathann}
in the $K$-finite realization; see also
\cite{ Sahi-crelle, gkz-ma}. 
Let $\mu \in (0,q)$.
There exists a $\mathfrak g^\sharp$-invariant (i.e.
$\mathfrak g^\sharp$ acts as skew Hermitian operators)
positive definite
inner product $(\cdot, \cdot)_{\nu}$
on the space $I_{K^\sharp}^\sharp(\mu)$
of $K^\sharp$-finite elements of the principal series
of $I^\sharp(\mu)$ (see \cite[Theorem 7.5]{Johnson-mathann}
and \cite[section 7]{BSZ}).
We will denote
the unitary representation of $G^\sharp$
on the completion
of $I_{K^\sharp}^\sharp(\mu)$
 as $(\mathcal C^\sharp(\mu), \pi^\sharp_\mu)$.

The non-compact picture has been further studied in \cite{BSZ}
and  we recall it briefly here.
We identify  $\bar N^\sharp$ with $\cX_{2q}$ via $Z \mapsto \exp \bar n^\sharp_Z$
and we consider the following
(formally defined) linear form
on the Schwarz space $\mathcal S(\mathcal X_{2q})$
$$
\bfZ_s(h)
= \gamma_{s,2q}
\int_{\mathcal X_{2q}} h(x) \ |\Pf (x)|^{s} dx,
\quad\quad
\gamma_{s,2q}= 
\frac{\pi^{\frac q 2 (s +2q-1)}}
{\prod_{j=0}^{q-1}\Gamma(\frac {s +2q-1} 2 -j)} \ ,
$$
for $s\in \bC$ with sufficiently large $\Re s$
where $\Pf$ denotes the Pfaffian polynomial.
This defines \cite{bopp-rub, BSZ}
 a family of  tempered distributions $\{\bfZ_s\}$
 which admits a holomorphic  continuation to 
the whole complex plane
and whose Fourier transform satisfies the following functional equality:
\begin{equation}
\label{fourier-zeta}
{\bfZ}_{2q,s-(2q-1)} (\mathcal F h)
=\bfZ_{2q,-s }(h),
\quad \quad  \ s\in \bC\ ;
\end{equation}
here the Fourier 
transform on $\mathcal X_{2q}$ is
given by: 
$$
\cF h(\zeta)
=
\int_{\Mat_{2q}^{ss}} h(x) e^{2\pi i (x,\zeta)} dx,
\quad\quad h\in \cS(\cX_{2q})\ ,
$$
The inner product in the $N^\sharp$-realization is given by:
$$
(f, g)_{\mu}:= (f,\bfZ_s*g)_{L^2(\mathcal X)}
= \bfZ_s\big (f*\check g\big )
\quad\mbox{with}\quad
s=\frac \mu q - (2q-1)
\ ,
$$
where $\check g: x\mapsto \bar f(-x)$, initially
defined on the Schwarz space
$\cS(\cX_{2q})$.
From \eqref{fourier-zeta}, 
we see 
that 
$(f, f)_{\mu}=\bfZ_{-\frac \mu q}(|\cF f|^2)$ for any $f\in \cS(\cX_{2q})$;
this makes sense because for $\mu \in (0,q)$, $\bfZ_{-\frac \mu q}$ is a locally integrable function.
The completion of 
$ \cS(\cX_{2q})$ is
\begin{equation}
  \label{eq:H-mu}
\mathcal C^\sharp(\mu)=
\{ f \in \cS'(\cX_{2q});
\, \ |\Pf|^{-\frac \mu {2q}} \cF f \in L^2(\cX_{2q})
\},
\ 
\end{equation}
(Note that the condition $ f \in \cS'(\cX_{2q})$
can be deduced also from 
$\ |\Pf|^{-\frac \mu {2q}} \cF f \in L^2(\cX_{2q})$.)

We summarize the results in the following
\begin{prop} 
\label{coml-sharp}
Suppose $\mu \in (0,q)$.
 There exists 
a $(\fg^\sharp, \pi^\sharp)$-invariant positive
definite inner product 
 $(\cdot, \cdot)_{\mu}$ on 
$I_{K^\sharp}^\sharp(\mu)
$. Its completion
defines a unitary representation
of $G^\sharp$. In the non-compact realization 
the Hilbert space is  described by \eqref{eq:H-mu}.
\end{prop}

The operator
\begin{equation}
\label{F-mu}
\mathcal F_{\mu}: f\in \cS(\cX_{2q})  
\longmapsto \phi= \cF^{-1} \big(|\Pf|^{-\frac \mu {2q}} \cF f \big)\in L^2(\cX_{2q}),
\end{equation}
then extends to a unitary operator from
$\mathcal C^\sharp(\mu)$ onto $L^2(\cX_{2q})$.
We denote 
the corresponding representation on 
 $L^2(\cX_{2q})$
by
\begin{equation}
\label{def_tau_sharp}
\tilde \pi_\mu^\sharp=\mathcal F_\mu
\pi_\mu^\sharp\mathcal F_\mu^{-1}
\ .
\end{equation}
We obtain a simple description of the action of $(\bar P^\sharp,
\tilde \pi_\mu^\sharp)$ (see \cite[sections 7 and 8]{BSZ}):
\begin{equation}
  \label{eq:p-sharp-act-1}
\tilde \pi_\mu^\sharp
(\exp \bar n^\sharp_Z)  \phi (W)=\phi(W -Z)  
\end{equation}
and for $g=me^{t\xi^\sharp}\in M^\sharp A^\sharp$ identified with an element of $\GL_{p+1}^+$:
\begin{equation}
  \label{eq:p-sharp-act-2}
\tilde \pi_\mu^\sharp \phi (W)=e^{-\rho^\sharp t} \phi(g^t W g)
\ .
\end{equation}

\medskip

\subsection{Principal series $I(\nu)$ of $G$
}\hfill\\ \vspace{-0.5em} 

Let 
$I(\nu)$ be the principal series 
induced from the following character of $P=\GL_p N$:
$$
\chi_\nu: l\,n
\longmapsto 
|\det (l)|^{\frac{\nu+\rho}{p}}
\ ,
$$
where $\rho=p^2/2$ is the half sum of the positive roots of $\ad \xi$,
with the similar condition as
in  (\ref{Isharp-mu}) and (\ref{Isharp-mu-}).
That is, the representation space is realized as Haar measurable functions $f(g)$ on $G$ satisfying
\begin{equation}
\label{I-nu}
f(g\, l\,n)=
|\det (l)|^{-\frac{\nu+\rho}{p}}
f(g) \ ,
\end{equation}
and
\begin{equation}
\label{I-nu-}
f{\big |}_{K}\in L^2(K) 
\quad\mbox{or equivalently}\quad
f{\big |}_{K}\in L^2(K/K\cap M)
\ ;
\end{equation}
the group $G$ acts 
on $I(\nu)$ by the  left regular action
and we denote this action by $(I(\nu),  \pi_\nu)$.
We denote
by $(I_K(\nu),  \pi_\nu )$ 
the space of $K$-finite
elements.

The homogeneous space $K/K\cap M$
can be realized as the Stiefel manifold
of rank $p$ isometries: 
$$
K/K\cap M=V_{p+1, p}=\{x\in \Mat_{p+1, p}; x^t x=I_p\}
\ ,
$$
where the group $K$ acts transitively via:
$$
K\ni \diag(k_1,k_2): 
\left\{
\begin{array}{ccl}
V_{p+1,p}&\longrightarrow& V_{p+1,p} \\
x &\longmapsto& k_1 x k_2^{-1}
\end{array}
\right.
\ ,
$$ 
and $K\cap M$ is the isotropy group of $
\begin{pmatrix}I_p\\ 0\end{pmatrix}\in V_{p+1,p}$.
The elements $f$ in $I(\nu)$ then satisfy
$$
f{\big |}_{K/K\cap M}\in L^2(K/ K \cap M)=L^2(V_{p+1, p}).
$$

We will need the multiplicity
free  decomposition of 
$L^2(V_{p+1, p})$
 under $K_0= \SO(p+1)\times {\rm SO}(p)$.
Let us recall that
each irreducible  representation of  $\SO(p)$
is determined by a $(q-1)$-tuples of integers
$$
\bn=(n_1, \cdots, n_{q-1}), \quad n_1\ge \cdots \ge n_{q-1}\ge n_{q-1}\ge 0\ ,
$$
and we write $\cW_{\bn}$ for the representation space.
Given a representation $\bm$ of
$\SO(p+1)$ 
we write $\bm
\succeq \bn$
 if $\bn$ appears
in the irreducible decomposition
of $\bm$ under $\SO(p)$.
It is a classical result, see e.g. \cite{Gelbart, Vinberg-Kimelfeld, Kobayashi-korea}
that  $\bn$ appears in  $\bm$
multiplicity free. 
This implies that
the space of $K_0$-finite elements of $L^2(V_{p+1, p})$ is
decomposed under $K_0$ as follows:
\begin{equation}
  \label{eq:L-2-K}
I_K(\nu) \sim L^2(V_{p+1, p})_{K_0}=\sum_{(\bm, \bn): \bm^\ast \succeq \bn
}
\cV_{\bm}\otimes \cW_{\bn}
\ ,
\end{equation} 
and this decomposition is multiplicity free.

\medskip

\subsection{The restriction map $R$}\hfill\\ \vspace{-0.5em}

We shall consider simply the restriction
of functions  in $I^\sharp(\mu)$
to $G\subset G^\sharp$. To clarify
its definition we note first that
the space $I_{K^\sharp}^\sharp(\mu)$
of $K^\sharp$-finite functions are smooth functions
on $G^\sharp$. Thus the restriction map 
$$
R: I_{K^\sharp}^\sharp(\mu)
\longmapsto  C^\infty(G), \quad
Rf(g)=f(g),\, g\in G
$$
makes sense. In the $K^\sharp$-realization of $I^\sharp(\mu)$,
we have $Rf\in L^2(K)$ for any $K^\sharp$-finite
elements $f\in L^2(\SO(2q))$.

Our main observation  is the following:

\begin{prop} 
\label{prop_main} 
Let $\nu,\mu\in \bC$ such that $\nu=\frac p {p+1}\mu$. 
The restriction map $R$ is a $G$-equivariant
isomorphism
from $I_{K^\sharp}^\sharp(\mu)
$ onto $I_K(\nu)$ in the sense
that 
$$
R\pi_{\mu}^\sharp (g) f=
\pi_{\nu}(g) Rf, \quad f\in I^\sharp(\mu), \quad g\in G
\ ,
$$
and it is unitary as a map 
from  
$I_{K^\sharp}^\sharp(\mu)\sim L^2(\SO(2q))_{K^\sharp}$
onto $I_K(\nu) \sim L^2(V_{p+1, p})_{K_0}$.
\end{prop}

\begin{proof}
Let $f\in I^\sharp_{K^\sharp}(\mu)$.
By Lemma \ref{lem_inclusion},
one check easily for $ln\in P\subset P^\sharp$
$\chi_\mu^\sharp (ln) 
=\chi_\nu(ln)$. Together with  \eqref{Isharp-mu}, it implies that $Rf$ satisfies \eqref{I-nu}. Moreover \eqref{Isharp-mu-} implies \eqref{I-nu-}.
So $Rf\in I(\nu)$ and 
$R\pi_{\mu}^\sharp (g) f=
\pi_{\nu}(g) Rf$ for any $g\in G$.
 As $f$ is $K^\sharp$-finite, $Rf$ is also $K$-finite.

The decompositions \eqref{eq:L-2-K}
and \eqref{eq:L-2-K-1} show the rest of the claim.
\end{proof}

Using Proposition \ref{coml-sharp}  we get 
that restriction to $G$ of the complementary series $\mathcal C^\sharp(\mu)$
defines a unitarizable representation of $G$,
which we write as 
$\mathcal C (\nu)$, whose $K$-finite
elements are the same as $I_K(\nu)$, 
by Proposition \ref{prop_main},
i.e,
$$
\mathcal C(\nu)=R\mathcal C^\sharp(\mu), \quad
\mathcal C_K(\nu)=I_K(\mu), \quad
\nu=\frac p {p+1}\mu, \quad  \mu\in (0,q).
$$ 

The main result of this paper is the following theorem which states that 
restriction $\mathcal C(\nu)
$ to the maximal parabolic subgroup
$\bar P$  of $G$ is irreducible.

\begin{thrm} 
\label{thm_main} 
Let $\nu=\frac p {p+1}\mu$ with  $\mu\in(0,q)$.
Then restriction to $G$ of the complementary
series $(\cC^\sharp(\mu), G^\sharp)$ 
defines a unitarizable irreducible representation
$\cC(\nu)$. It
is the unitarization of the 
principal series representation $(I_K(\nu), G)$ realized in the non-compact picture.
Moreover, 
its remains irreducible when restricted to 
the maximal parabolic subgroup $\GL_p \, \bar N=\bar P$.
\end{thrm}

The irreducibility 
under $G_0$ in
the above statement is
essentially proved in \cite{Lee-Loke}.
Indeed let $\tilde K=\text{Spin}(p+1)
\times \text{Spin}(p)$.
The representation
$\mathcal C (\nu)$
 is treated
as representation of 
$\text{Spin}(p+1, p)$ and it is proved
\cite[12.2.1]{Lee-Loke}
that the $(\fg, \tilde K)$-module
of $\mathcal C (\nu)$
 is irreducible.
However the representations of $\tilde K$
in $\mathcal C (\nu)$ descend
to the representation
of $K_0$ and thus the 
$(\fg, \tilde K)$-module
is the same as $(\fg, K)$-module
$\mathcal C_K (\nu)$, and the latter is
then irreducible.

To prove the rest of Theorem \ref{thm_main}, we will use the non-compact picture.
As $\bar P\subset \bar P^\sharp$ we can
find the action of $\bar P$ on 
$(\tilde \pi_\mu^\sharp,L^2(\cX_{2q}))$:
\begin{lem}
\label{lem_pi}
The representation of $\bar P$ on
$(\tilde \pi_\mu^\sharp,L^2(\cX_{2q}))$ is unitarily equivalent to the representation 
$(\pi, L^2(N_p))$ given by:
\begin{equation}
  \label{eq:fou-n-act}
{\pi}
(\bar n_0 )
\cdot \phi (\bar n)=\phi (\bar n_0^{-1} \bar n)  
\end{equation}
for an element $n_0\in \bar N$
and for an element $g\in \GL_p$:
\begin{equation}
  \label{eq:fou-ma-act}
{\pi}
(g) \phi( \bar n)=|\det g|^{-\frac p2}\phi(g^{-1} \cdot  \bar n)
\ ,
\end{equation}
 where the action of $\GL_p$ on $\bar N$ is given by 
(\ref{ma-on-n}).
\end{lem}

\begin{proof}
Let us consider the unitary isomorphism:
$$
\begin{array}{rcl}
L^2(N_p)&\longrightarrow&L^2(\cX_{2q})
\\
\psi & \longmapsto& \phi
\end{array}
\quad\mbox{given by} \quad
\psi(z,v)=\phi(M(z,v))
\ .
$$
It is easy to check \eqref{eq:fou-ma-act}.
Now for $(z_o,v_o)\in N_p$,
using 
\eqref{relation_n_nsharp},
we compute:
\begin{eqnarray*}
\left(\exp \bar n^\sharp_{M(z_o,v_o)} m_h^\sharp\right) 
\cdot \phi \big( M(z,v)\big)
=
\phi(-M(z_o,v_o) + h^t M(z,v) h)
\end{eqnarray*}
and direct computations show
$$
-M(z_o,v_o) + h^t M(z,v) h = M\left((z_o,v_o)^{-1} (z,v)\right)
$$
so the action of $\bar N$ is given by \eqref{eq:fou-n-act}.
\end{proof}

So by Schur's Lemma, 
Theorem \ref{thm_main} is proved once we have shown the following proposition:
\begin{prop}
\label{prop_schur}
Let $T$ be a bounded operator on $L^2(\bar N)$ commuting with the action 
$\pi$ of $\bar P$
defined in Lemma \ref{lem_pi}.
Then $T$ is the scalar multiple of the identity.
\end{prop}

\noindent {\it Remark:} 
Using the non-compact pictures, 
one can show the unitarity of $C(\nu)$.
Indeed using Lemma \ref{lem_inclusion}
it is not difficult to show that the Knapp-Stein intertwiner 
$A_\mu^\sharp$ and $A_\nu^\sharp$
for the series $I^\sharp(\mu)$ and $I(\nu)$
satisfies:
\begin{equation}
\label{intertwiners}
A_\mu^\sharp(f^\sharp)(\exp \bar n^\sharp_{M(z,v)})
=
A_\nu(f^\sharp|_G)(\exp \bar n_{(z,v)})
\ ,
\end{equation}
and the properties of $A_\mu^\sharp$ described in  \cite{BSZ}
imply that $I_\nu$ is unitarizable.

Knapp-Stein intertwiners are (nilpotent) convolution operator with very singular kernels. 
By \cite{BSZ}, $A_\mu$ is an abelian convolution with 
 a power of the Pfaffian.
It can be computed using geometric means that the kernel of $A_\nu$ is of the form $Q(z,v)^s=\det (z+\frac 12 vv^t)^s$
 for $(z,v)\in  \bar N$ and it is only through some elementary but tricky matrix computations that it can be linked with the abelian convolution with some power of the Pfaffian as in \eqref{intertwiners}.

\medskip

\section{Proof of Proposition \ref{prop_schur}}

We will use some well-known results
for the Plancherel formula
and von Neumann algebras of
left regular representations
of $\bar N$ on $L^2(\bar N)$; see \cite[Chapt.14]{Wallach} and \cite{Dixmier}
for general locally compact groups
\cite{Corwin-Gr}
for the case of nilpotent groups.
We describe first
the support
of the Plancherel measure described in \cite{BJR, mythesis, stri}.

\subsection{The support of the Plancherel measure}

To ease notation we write elements
in $\bar\fn$  or $\bar N$ as $n$
instead of $\bar n$.  
We may identify them with elements of $\cX_{2q}$
by \eqref{M(z,v)};
on $\cX_{2q}$ we consider the standard inner product $(Z,W)=\frac 12 \Tr ZW^*$. Hence we have equipped $\fn$ with an inner product
and we can now identify the dual  $\fn^\ast$
with $\fn$. The dual action
of $g\in \GL_p$ on $\fn^\ast=\fn$
will be written as $g \ast n$. 

We fix a generic point
 $o_{\fn^*}=(z_{o_{\fn^*}}, v_{o_{\fn^*}})$,
 the element of $\fn^*\sim \fn$ defined 
using \eqref{M(z,v)} by:
$$
M(o_{\fn^*})=J_q
\quad\mbox{where}\quad
J_q=\diag (J,\ldots,J)\in M_{2q}
\quad\mbox{and}\quad
J= \begin{pmatrix}
0&1 \\
-1&0
 \end{pmatrix}
  \ .
$$
 
 It is easy to see that a representative of a coadjoint orbit can be chosen of the form
$(z, v)\in \fn_p$ with $zv=0$, 
that is, the vector $v$ being in the kernel of the matrix $z$.
Let $\cO$ be the collection of those representatives $(z,v)$ 
with $M(z,v)$ non-singular:
$$
\cO:=\{ (z,v)\in \fn_p^*\ , \ zv=0 \quad\mbox{and}\quad \det M(z,v)\not=0
\}
\ .
$$ 
It is easy to see that $\cO$ is the following union of $\SO(p)$-orbits of certain "diagonal representatives":
$$
\mathcal O:
= \SO(p) \Delta \ast o_{\fn^*}
\quad\mbox{where}\quad
\Delta=\{\diag ( d_1 I_2,\ldots, d_{q-1} I_2, d_q) \ , \ d_j\in \bR^*\}\subset\GL_p 
\ .
$$

Any $w\in \mathcal O
\subset \fn^\ast$ then induces
an irreducible unitary representation $\lambda_w$
of $\bar \fn$ and $\bar   N$
on $L^2(\bR^{q-1})_w\sim L^2(\bR^{q-1})$.
We describe the representation of $\fn$
 very briefly  for 
the element $o_{\fn^*}= (z_{o_{\fn^*}}, v_{o_{\fn^*}})$.
The construction for general $w\in 
\cO$ can be done similarly by using the equivariant
action of $\GL_p$ on $\fn^\ast$.
As the writing of $w\in \cO$ as $\GL_p \cdot o_{\fn^*}$ is not unique,
there is a certain ambiguity here but it is harmless for our proof.

The element $o_{\fn^*}$ defines
a splitting (or a complex structure) of
$\bR^{2(q-1)} =
\bR^{q-1} +
\bR^{q-1}$.
The space $\fn$ is decomposed as
\begin{equation}
\label{n-deco}
\fn =
\cX_p
 \oplus \bR^p  
=
\fn_0 \oplus \fh
\end{equation}
where 
$\fn_0:=
(z_{o_{\fn^*}}^{\perp}\cap \cX_p)  \oplus
\bR v_{o_{\fn^*}}$
while $\fh:=\bR z_{o_{\fn^*}}
+ \bR^{q-1} +
\bR^{q-1} $ is the  
Heisenberg algebra. 
There exists a unique representation 
$(\lambda_{o_{\fn^*}}, L^2(\bR^{q-1}))$
of $N$ 
whose restriction to $\exp \bR o_{\fn^*}$ is given 
by the character
$ \exp i2\pi o_{\fn^*}$.

The Plancherel formula for $L^2(N)$ is given by
\begin{equation}
  \label{eq:pl-fo}
\Vert f\Vert^2_{L^2(\bar N)}
=\int_{\mathcal O } \Vert \hat f (w)
\Vert_{2}^2 d\iota(w),
\quad \quad
f(0)
=\int_{\mathcal O } 
 \Tr(\hat f (w))
 d\iota(w) \ ,  
\end{equation}
where we have denoted the group Fourier transform of a function $f$ by
$$
\hat f(w)=\int_N f(g)\lambda_w(g) dg
\ ,
$$
and its Hilbert-Schmidt norm by $\Vert \hat f (w)
\Vert_{2}$
i.e. in $HS_w:=L^2(\bR^{q-1}_w)
\otimes L^2(\bR^{q-1}_w)^\ast
$.
The Plancherel measure $d\iota$ can
be explicitly computed but we will not need
it here.
In otherwords, the regular action of $N\times N$
on $L^2(N)$ is decomposed as
\begin{equation}
  \label{eq:bi-deco}
L^2(N) = \int_\cO \lambda_w \otimes \lambda_w^\ast d \iota (w)
\ ,  
\end{equation}
where $\lambda_w^\ast$ is the contragradient
of $\lambda_2$.
\subsection{Proof of Proposition \ref{prop_schur}}
\label{subsec_proof_prop_schur}
Let $\Omega:=\{(z, v)\fn; \ \det(M(z, v))\ne 0\}$.
Clearly 
$\Omega$ is open and dense in $\fn$, and
which strictly contained in $\Omega\subsetneq \cO$.
By elementary matrix computations,
it can be also described as the orbit of $o_{\fn^*}$:
\begin{lem}
\label{lem_omega} $GL_p$ acts transitively on $\Omega$
and we have
$\Omega=\GL_p\cdot o_{\fn^*}=
\GL_p / \Sp(q-1, \bR)$.
\end{lem}

\begin{proof}
Let $(z, v)\in \Omega$.
The diagonalization
of $z$ provides a $g\in SO(p)$ such that
$g\cdot (z, v)=(w, u)$ with $u=(u_1, \cdots, u_p)$
and
$$
w=\diag (\begin{pmatrix}
0 & w_1\\
-w_1 &0 \end{pmatrix}, \cdots,
\begin{pmatrix} 0 & w_{q-1}\\
-w_{q-1} & 0 \end{pmatrix})
\ .
$$
We compute 
on the one hand 
$\det M(w, u)= (w_1\cdots w_{q-1} u_p)^2$
and on  the other hand
$\det 
=\det M(g(z, v))
=\det g \det M(z, v) =
 \det M(z, v) \ne 0$. Namely $u_p, w_1, \dots, w_n \ne 0$. 
We solve now the equation 
$g(w, u)=o_{\fn^*}$ with $g\in GL_p$
of the form
$g=\begin{pmatrix}
A &0\\
B & c\end{pmatrix}$
viz,
$$
AZ_1A^t=J_0
\quad ,\quad 
AZ_1B^t+ AYc=0,
\quad\quad 
 cu_p=1
 \ ,
$$
which is easy to see to have a solution, e.g.
by taking $c=v_p^{-1}$, $B=-v_p^{-1}Z_1^{-1} Y$ and
$$
A=\diag(\sgn w_1|w_1|^\frac 12,
|w_1|^\frac 12 , \cdots, 
\sgn w_{q-1}|z_{q-1}|^\frac 12,
|z_1|^\frac 12) 
\ .
$$

The isotropic subgroup in $\GL_p$
of $o_{\fn^*}$
consists
of $g=\begin{pmatrix}
A &C\\
B & D\end{pmatrix}\in GL_p$
with $B=0, C=0, D=1$ and $A$ such that
$A J_{q-1} A^t=J_{q-1}$. Thus it is a realization of the symplectic group
$\Sp(q-1, \bR)$ and we have
$\Omega=\GL_p / \Sp(q-1, \bR)$.
\end{proof}

There exists \cite{folland} 
a representation 
$(\tau, L^2(\bR^{q-1}), 
\Mp(q-1, \bR))$
 of  the double
cover $\Mp(q-1, \bR)$ of $\Sp(q-1, \bR)$ such that
\begin{equation}
 \label{metap}
\tau (\tilde g) \lambda_{o_{\fn^*}}(n)
\tau (\tilde g)^\ast 
= \lambda_{o_{\fn^*}} (g\cdot n),
\quad \quad
\tilde g\in \Mp(q-1, \bR)\ ,
\end{equation}
with $\tilde g\in
 \Mp(q-1, \bR) \mapsto g\in \Sp(q-1, \bR)
$ the double covering.
Furthermore the representation 
$(\tau, Mp(q-1, \bR), L^2(\bR^{q-1}))$ is
a sum of
two irreducible inequivalent
representations \cite[Theorem 4.56]{folland}, 
\begin{equation}
 \label{ev-od}
L^2(\bR^{q-1})=L^2_0(\bR^{q-1})\oplus 
L^2_1(\bR^{q-1})
\end{equation}
of even an odd functions.

\medskip

We can now prove Proposition \ref{prop_schur}.
\begin{proof}[Proof of Proposition \ref{prop_schur}]
Let $T$ be a bounded operator on $L^2(N)$ commuting with the action 
$\pi$ of $\bar P$
defined in Lemma \ref{lem_pi}.

As $T$ commutes with the left translation,
by the Plancherel Theorem \cite{Dixmier},
 there exists a measurable field $\{\hat T(w), w\in \cO\}$
of bounded operators on $L^2(\bR^{q-1}_w)$
such that
\begin{equation}
\label{nil-for-def}
\widehat{Tf}(w) =
 \hat f(w) \hat T(w),
 \quad \quad f\in \cS(N) \ ;
\end{equation}
this measurable field of operator is unique (up to a $\iota$-negligible set).

Let $w= g \ast o_{\fn^*}$ with $g\in \SO(p)\Delta$. 
By the orbit method, the representations $\lambda_w$ 
and $n\mapsto \lambda_{o_{\fn^*}} (gn)$ are unitarily equivalent:
there exists a unitary operator $A_w$ such that 
$A_w\lambda_w (n)=  \lambda_{o_{\fn^*}} (gn) A_w$
and, for $f\in\cS(N)$, we compute with the change of variable $n_1=gn$:
\begin{eqnarray}
\hat f(w)
&=& \int_N f(n) A_w^{-1} \lambda_{o_{\fn^*}} (gn)  A_wdn
=  A_w^{-1} \int_N \big(\pi (g)\big) f(n_1) \lambda_{o_{\fn^*}} (n_1) dn_1 A_w \nonumber\\
&=&
A_w^{-1} \widehat{\pi (g)f}(o_{\fn^*})A_w \label{hatf_Aw}
\ .
\end{eqnarray}
Now as $T$ commutes with $\pi(g)$, 
we obtain easily:
$$
\widehat{Tf} (w) = 
A_w^{-1} \left(T \big(\pi (g)f\big)\right)\widehat{\ }
\, (o_{\fn^*})A_w \ .
$$
Using \eqref{nil-for-def} and the uniqueness of $\{\hat T(w), w\in \cO\}$,
we obtain for $\iota$-almost all $w\in \cO$,
\begin{equation}
\label{hatT_Aw}
\hat T (w) =
A_w^{-1} \hat T(o_{\fn^*}) A_w 
\ .
\end{equation}
We may assume that $\hat T(o_{\fn^*}) $ exists and that  relation \eqref{hatT_Aw} holds for all $w\in \cO$. 

In the same way, we consider $\tilde g\in \Mp(q-1, \bR)$
and the equivalence relation \eqref{metap}.
Proceeding  just as above, we obtain:
\begin{equation}
\label{hatT_tau}
\hat T(o_{\fn^*})=
\tau(\tilde g)\hat T(o_{\fn^*}) 
\tau(\tilde g)^{-1}
\ .
\end{equation}
It follows then from the irreducible decomposition  (\ref{ev-od})
that
$\hat T(o_{\fn^*})$ is constant on each space, namely 
$$\hat T(o_{\fn^*}) = c_0 I  + c_1 U
$$
where  
$U$ is
 the reflection,
$$
U h(x) =h(-x)
\quad, \quad h\in L^2(\bR^{q-1}) \ , \ x\in \bR^{q-1}
\ .
$$

By the Plancherel Theorem \cite{Dixmier},
there exists a bounded operator $T_1:L^2(N)\rightarrow L^2(N)$ 
which commutes with the left translations and satisfies
\begin{equation}
\label{def_T1}
 \widehat {T_1f} (w) =  \hat f (w)A_w^{-1} U A_w
\ , 
\end{equation}
  for all $w\in \cO$ and $f\in \cS(N)$.
Because of \eqref{hatT_Aw} and \eqref{hatT_tau},
we have $T=c_0 I +c_1 T_1$.
So the proof of Proposition \ref{prop_schur} will be over once we have shown that $c_1=0$
and for this it suffices to show that $T_1$ does not commute with the action of $(\pi, \GL_p)$.

As $T_1$ is bounded on $L^2(N)$ and commutes with left translation,
it is a convolution operator with a tempered kernel:
there exists $\kappa \in \cS'(N)$ such that $T_1f=f*\kappa$ for any $f\in \cS(N)$.
We claim that $\kappa$ is not invariant under $\GL_p$
and this shows that $T_1$ does not commute with the action of $(\pi, \GL_p)$.

To show our claim, 
we first compute  $T_1f(0)$
for $f\in \cS(\bar N)$.
For this we will use the Plancherel formula  (see \eqref{eq:pl-fo}):
\begin{equation}
\label{T1f0}
T_1f(0)
=\int_{\cO} \Tr (\widehat{T_1f}(w))
\ .
\end{equation}
Now by \eqref{hatf_Aw} and \eqref {def_T1},
for $w= g\ast o_{\fn^*}$,
 we have:
 \begin{equation}
\label{tr_T1fw}
\Tr \left(\widehat{T_1f}(w)\right)=
\Tr \left(\widehat{f}(w) A^{-1} U A_w\right)=
\Tr \left(A_w\widehat{f}(w) A^{-1} U \right)=
\Tr \left(\widehat{\pi(g) f}( o_{\fn^*})  U \right)
\ .
\end{equation}
So we just want to compute 
the trace of 
$\widehat{ f}( o_{\fn^*})  U$
 on the Hilbert space
$L^2(\bR^{q-1})$.
This can be derived from the standard
formulas (\cite[Chapt. XII, \S6]{Stein}, \cite[Chapt. II, \S2-\S3]{veluma}) 
for the Weyl transform.

Indeed
considering the decomposition (\ref{n-deco}),
we write the elements of $\fn$ as $h+h^\perp$ where $h\in \fh$ and $h^\perp\in \fn_0$. 
Integrating $f$ over $\fn_0$, we obtain the function $F$ with $\fh$:
$$
F(h)=
\int_{\fn_0} f(h+h^\perp)  dh^\perp.
$$
We now identify $\fh$ with the Heisenberg group and we write  the elements of $\fh$ as $h=(x,y,t)\in \bR^{q-1}\times\bR^{q-1}\times\bR$.
It is clear that $\widehat{ f}( o_{\fn^*})$ coincides with the Schr\"odinger representation of the Heisenberg group at $F$. 
From \cite[Chapt. XII, \S6.3]{Stein}, this shows that
$\widehat{ f}( o_{\fn^*})$ is the integral operator on $L^2(\bR^{q-1})$ with kernel
$$
K_\fh(x,y):=c\int_{\bR^{q-1} \times \bR} 
e^{2\pi i (\frac 12 u\cdot (y+x) +  t)}
F(u, y-x,t) du dt
\quad, \quad x,y\in \bR^{q-1},
$$
where $c=c_q$ is a known constant
(our $t$ here corresponds to $\frac 14 t$
in \cite[Chapt.~XII, \S6.3, (91)]{Stein}).
So the kernel of 
the operator $\widehat{ f}( o_{\fn^*})  U$ is
$K_\fh (-x,y)$ and we can now compute the trace of the operator:
\begin{eqnarray*}
\Tr\left(\widehat{ f}( o_{\fn^*})  U \right)
&=&
\int_{\bR^{q-1}} K_\fh(-x,x) dx
\\
&=&
c\int_{\bR^{q-1}}\!
\int_{\bR^{q-1} \times \bR\!\!\!\!
} 
e^{2\pi i  t }
F(u, 2x,       t) du dtdx
\ .
\end{eqnarray*}
Thus we have obtained:
$$
\Tr\left(\widehat{f}( o_{\fn^*})  U \right)
 = C (\cF f)(o_{\fn^*}) \ , 
$$
for some  non-zero known constant $C$
where we have denoted by $\cF$ the Euclidean Fourier transform on $\fn$,
that is,
$$
\cF f(\zeta ,\nu) = \int_{\fn} f(z,v) e^{2i\pi (\langle z,\zeta\rangle +\langle v,\nu\rangle)} dz dv
\ ,
$$
where we have used the canonical Euclidean scalar products on $\mathcal X_{q-1}$ and $\bR^{q-1}$.

Now by \eqref{tr_T1fw},
this shows that for any $w\in \cO$, we have:
$$
\Tr\left(\widehat{f}( w)  U \right)
 = C (\cF f)(w)
 \ .
$$
By  \eqref{T1f0},
we obtain:
$$
\int f(n) \kappa (n^{-1}) dn = 
T_1f(0) = C \int_\cO \cF (w) d\iota (w)
\ .
$$
Hence the support of $\cF \kappa (\cdot ^{-1})$ is included in 
$\overline{\cO}$.
But $\overline{\cO}$ is invariant under $n\mapsto n^{-1}$ but not invariant under $\GL_p$ since $\cO$ is strictly included in $\Omega=\GL_p\cdot o_{\fn^*}$
(see Lemma \ref{lem_omega}).
So $\kappa$ is not $\GL_p$-invariant.
This shows that $T_1$ does not commute with $(\pi,\GL_p)$
and concludes the proof of Proposition \ref{prop_schur}.
\end{proof}

\subsection{Decomposition of $L^2(N)$ under $\bar N\times \Sp(q-1, \mathbb R)$}
We note that the above proof also yields
a decomposition of $L^2(N)$ under the action
of $\bar N\times \Sp(q-1, \mathbb R)$.
Consider first the reference point $w=o_{\fn^*}$
and the space $L^2(\mathbb R^{q-1})\otimes
L^2(\mathbb R^{q-1})$ with $N$ acting
on the left factor by $\lambda_w$. Note
that the metaplectic representation $\tau$
on $L^2(\mathbb R^{q-1})$, by its definition,
defines a unitary representation of $\tau\otimes \tau^\ast$
of $\Sp(q-1, \mathbb R)$ on 
$L^2(\mathbb R^{q-1})\otimes
L^2(\mathbb R^{q-1})$,
 viewed as Hilbert-Schmidt operators, by
$$
(\tau\otimes \tau^\ast)(g)
T=\tau(g) T \tau^\ast(g) \ .
$$
The group $\bar N\times \Sp(q-1, \mathbb R)$ acts on 
 $L^2(\mathbb R^{q-1})\otimes L^2(\mathbb R^{q-1})$
 and we denote the corresponding representation, 
 by
$\lambda_\omega\rtimes \tau_\omega$.

Using the decomposition of $L^2(\bR^{q-1})$ into even and odd functions $L^2(\bR^{q-1})_{i}$, $i=0,1$, 
we have:
$$
L^2(\mathbb R^{q-1})\otimes
L^2(\mathbb R^{q-1}) = 
L^2(\mathbb R^{q-1})\otimes
L^2(\mathbb R^{q-1})_0 \oplus
L^2(\mathbb R^{q-1})\otimes
L^2(\mathbb R^{q-1})_1, 
$$
and we obtain the $\bar N\times \Sp(q-1, \mathbb R)$-irreducible decomposition:
$$
\lambda_\omega\rtimes \tau_\omega=
(\lambda_\omega\rtimes \tau_\omega)_0 +(\lambda_\omega\rtimes \tau_\omega)_1 \ .
$$
Clearly this
construction can be done for any $\omega$. We have then
\begin{cor} 
The space $L^2(N)$ is decomposed 
under $\bar N\times \Sp(q-1, \mathbb R)$
as 
$$
L^2(N) = \int_\cO 
(\lambda_\omega\rtimes \tau_\omega)_0 +(\lambda_\omega\rtimes \tau_\omega)_1 \
d \iota (w)\ .
$$
\end{cor}

\proof[Acknowledgements]
We would like to thank
 Professors Pierre Cartier, Jacques Faraut and Jean Ludwig
for some fruitful discussions. We are grateful
to Professor Robert Stanton for drawing
us attention to the work \cite{BK}.

\end{document}